\documentclass{article}
\title {An inductive approach to constructing universal cycles on the $k$-subsets of $[n]$}  
\author{Yevgeniy Rudoy \\ Johns Hopkins University}
\usepackage{amsmath}
\usepackage{verbatim}
\usepackage{parskip}
\usepackage{amsthm}
\usepackage{amssymb}
\usepackage{etoolbox}
\usepackage{amsfonts}
\usepackage[mathscr]{euscript}

\makeatletter
\newtheoremstyle{nameStyle}
  {1em}{\topsep}
  {\sffamily }{}
  {\bf \sffamily }{:\mbox{$ $}}
  {.5em} {\thmname{\@ifempty{#3}{#1}\@ifnotempty{#3}{#3}}}
\makeatother
\theoremstyle{nameStyle}
\newtheorem*{named}{This Theorem Should Have a name!}

\newtheoremstyle{sf}
  {1em}{\topsep}
  {\sffamily }{}
  {\bf \sffamily }{:\mbox{$ $}}
  {.5em}{}
\newtheoremstyle{remark}
  {1em}{1em}
  {\sffamily }{}
  {\bf \sffamily }{:\mbox{$ $}}
  {.5em}{}
\newtheoremstyle{note}
  {\topsep}{\topsep}
  {\scshape}{}
  {\scshape}{:\mbox{$ $}}
  {.5em}{}
\theoremstyle{sf}
\newtheorem{lem}{Lemma}[section]
\newtheorem{theo}[lem]{Theorem}
\newtheorem{cor}{Corollary}[lem]
\newtheorem*{conjecture}{Conjecture}

\theoremstyle{remark}
\newtheorem*{rem}{Remark}

\theoremstyle{note}
\newtheorem*{note}{NOTE}

\newcommand{\SUB}{\textnormal{SUB}} 
\newcommand{\WEAVE}{\textnormal{WEAVE}}

\newcommand{\Cbar}{\hspace{1pt}\overline {\hspace{-1pt}  \raisebox{0pt}[7pt]{$C$} \hspace{-1pt}} \hspace{1pt}}
\newcommand{\Deltabar}{\hspace{1pt}\overline {\hspace{-1pt}  \raisebox{0pt}[7.5pt]{$\Delta$} \hspace{-1pt}} \hspace{1pt}}
\newcommand{\lcm}{{\textnormal{lcm}}}

\newcommand{\osum}[1]{{\bigoplus}^{\hspace{-2pt}k}}

\begin{document}
\maketitle
\begin{abstract}
In this paper, we introduce a method of constructing universal cycles on sets by taking ``sums'' and ``products'' of smaller cycles. We demonstrate this new approach by proving that if there exist universal cycles on the 4-subsets of [18] and the 4-subsets of [26], then for any integer $n\ge18$ equivalent to $2 \pmod{8}$, there exists a universal cycle on the 4-subsets of [n].
\end{abstract}

\section{Introduction}

%N.G. de Bruijn, A Combinatorial problem, Nederl. Akad. Wetensch, Proc. 49 (1946)
Consider the binary sequence $00011101.$
If we regard this sequence as a cycle, each of the 8 binary triples appears exactly once as a block of consecutive symbols in our sequence. In 1946, de Bruijn \cite{B} showed that for any $n$ and $k$, there exists an $n$-ary sequence in which each $n$-ary $k$-tuple appears exactly once. Such sequences are now known as \emph{de Bruijn cycles}.

In 1992, Chung, Diaconis, and Graham \cite{C} explored various generalizations of de Bruijn cycles, which they called \emph{universal cycles} or \emph{ucycles}. One such generalization was to universal cycles on 
$\left[\begin{smallmatrix} n \\ k \end{smallmatrix} \right]$\footnote{Here, $\left[\begin{smallmatrix} n \\ k \end{smallmatrix} \right]$ denotes the set of all $k$-element subsets of $[n] = \{0,1,\ldots, n-1\}$}: $n$-ary sequences in which each block of $k$ consecutive symbols consists of $k$ different symbols, and any set of $k$ symbols chosen from $[n] = \{0,1,\ldots, n-1\}$ is represented exactly once as a set of $k$ consecutive symbols in the sequence.

Chung, Diaconis, and Graham \cite{C} proved that for universal cycles on $\left[\begin{smallmatrix} n \\ k \end{smallmatrix} \right]$ to exist, it is necessary for $k$ to divide $ \binom{n-1}{k-1}$, a result reproduced below:

\begin{lem}[Chung, Diaconis, and Graham]
$k \big| \binom{n-1}{k-1}$ is a necessary condition for the existence of a universal cycle on $\left[\begin{smallmatrix} n \\ k \end{smallmatrix} \right]$.
\end{lem}
\begin{proof}
Let $C$ be a universal cycle on $\left[\begin{smallmatrix} n \\ k \end{smallmatrix} \right]$, and let $s$ be any symbol in $[n]$. For each occurrence of $s$ in $C$, there will be exactly $k$ different blocks of size $k$ which contain that occurrence of $s$. Since no block can contain multiple occurrences of $s$, the total number of blocks containing $s$ must be $k$ times the number of occurrences of $s$.
As there is exactly one such block for each set of $k$ symbols in $[n]$ containing $s$, $k$ must divide $\binom{n-1}{k-1}$.
\end{proof}

Chung, Diaconis, and Graham also conjectured that for any $k$, provided that $n$ was sufficiently large, this necessary condition was also sufficient. In other words,

\begin{conjecture} [Chung, Diaconis, and Graham]
 Uycles exist for $\left[\begin{smallmatrix} n \\ k \end{smallmatrix} \right]$ provided that $k$ divides $\binom{n-1}{k-1}$ and $n \ge n_0(k)$.
\end{conjecture}

It is easy to show that this conjecture holds when $k\in\{1,2\}$.

In 1993, Jackson \cite{J}  showed that for all $n\ge8$ not divisible by 3, there exist ucycles on $\left[\begin{smallmatrix} n \\ 3 \end{smallmatrix} \right]$, completing the $k=3$ case. The same paper also proved that for odd $n\ge9$, there exist ucycles on $\left[\begin{smallmatrix} n \\ 4 \end{smallmatrix} \right]$. Since  $\binom{n-1}{3}$ is divisible by 4 if and only if $n$ is odd or $n \equiv 2 \pmod 8$, this leaves only the $n \equiv 2 \pmod 8$ case unresolved for $k=4$. 

In 1994, Hurlbert \cite{H} unified Jackson's  results and gave a partial solution for $k=6$ with the following theorem:

\begin{theo}[Hurlbert]
For $k \in \{3,4,6\}$ and sufficiently large $n$ relatively prime to $k$, there exist ucycles on $\left[\begin{smallmatrix} n \\ k \end{smallmatrix} \right]$. 
\end{theo}

In addition to the published results above, Jackson claims to have an unpublished result completing the $k=4$ and proving the $k=5$ case. 

In this paper, we provide a new method of constructing universal cycles on $k$-subsets of $[n]$. Instead of finding a ucycle directly, we build the ucycle up from smaller cycles. In particular, we demonstrate a method for taking ``sums'' and ``products'' of cycles. Although these methods have significant limitations, they give us a powerful new tool for finding universal cycles on sets. In fact, an application of these new techniques allows us to prove the following results:

\begin{named}[Main Theorem]
If $a$ and $b$ are positive multiples of 8 such that  neither $a+1$ nor $b+1$ are divisible by 3, then if there exist  universal cycles on $\left[\begin{smallmatrix} a+2 \\ 4 \end{smallmatrix} \right]$ and $\left[\begin{smallmatrix} b+2 \\ 4 \end{smallmatrix} \right]$, there must exist  universal cycles on $\left[\begin{smallmatrix} a+b+2 \\ 4 \end{smallmatrix} \right]$.
\end{named}
\begin{named}[Corollary to Main Theorem]
As long as we can find universal cycles on $\left[\begin{smallmatrix} 18 \\ 4 \end{smallmatrix} \right]$ and $\left[\begin{smallmatrix} 26 \\ 4 \end{smallmatrix} \right]$, we can find universal cycles on $4$-subsets on $\left[\begin{smallmatrix} n \\ 4 \end{smallmatrix} \right]$ for any $n = 2 \pmod 8$ satisfying $n\ge18$.
\end{named}

\section{Definitions}\label{definitions}

\subsection{General}
\begin{note}
In this paper, $\cup$ and ``union" of two multisets $A$ and $B$ will be used to denote the multiset which consists of combining the elements without removing any duplicates. For example, we would say $$\{a, a, b\} \cup\{a, b, c\} = \{a, a, a, b, b, c\}.$$
We will not be using the standard set union in this paper. 
\end{note}
Let $[n] = \{0,1,\ldots, n-1\}$ and $\left[\begin{smallmatrix} n \\ k \end{smallmatrix} \right]$ denote the set of all $k$-element subsets of $[n]$. Note that this may differ from some conventional definitions of $[n] = \{1, 2, \cdots n\}$. 

Define a \textbf{k-string} to be a string of length $k$, and a \textbf{k-multiset} to be a multiset of cardinality $k$.

Denote the cardinality of a multiset $A$ as $|A|$ and the length of a string $S$ as $|S|$.

Define the powerset of a set $A$, denoted $P(A)$, to be the set of subsets of $A$. Furthermore, define $P_k(A) = \{M\in P(A) : |M| = k\}$ to be the set of all $k$-element subsets of $A$.

If $M$ and $N$ are both multisets of multisets, define their product,
$M \times N$, to be the multiset consisting of all the unions of elements of $M$ with elements of $N$. In other words, $M \times N= \{A \cup B: A\in M, B\in N\}$. For example, if $M = \{\{a,b\}, \{c\}  \}$ and $N = \{\{x\}, \{y,z\} \}$, then
$$N \times M = \{\{a,b,x\}, \{a,b,y,z  \},  \{c,x\}, \{c,y,z  \}  \}.$$

If $S$ and $T$ are both strings, let the concatenation of $S$ with $T$, written $S\cdot T$, be the string consisting of the characters of $S$ followed by the characters of $T$.

Denote the multiset of $k$-substrings of $S$ as $\SUB^k(S)$. For example, if $S = abcabcd$, and $k=3$, we would have
$\SUB^k(S) = \{abc, bca, cab, abc, bcd\}.$

If $S$ is a string, let $\Gamma (S)$ denote the multiset of characters in $S$. If $M$ is a multiset of strings, then let $\Gamma (M)$ denote $\{\Gamma (S) : S \in M\}$. For example,
$\Gamma(cycle) = \{c, c,e, l, y  \}$
and
$\Gamma(\{and, text\}) = \{a, d, e, n, t, t, x\}.$

\subsection{Cycles} 
Let a \textbf{length z} cycle be a string of length $z$. 
\\If $C$ is a cycle, let $C_x$ denote the $(x+1)$\textsuperscript{th} symbol in $C$, up to modulo $|C|$. Note that the first symbol of $C$ is $C_0$ and not $C_1$.
\\If $C$ is a cycle, let $C_x^k$ denote the $k$-string $C_xC_{x+1}\cdots C_{x+k-2}C_{x+k-1}$.
\\If $C$ is a cycle, let the \textbf{k-range} of $C$ be the multiset ${\{C_x^k : 0 \le x \le |C| - 1 \}}$. We will use $R^k(C)$ to denote the $k$-range of $C$. For example, $$R^2(inoh) =  \{in, no, oh, hi  \}$$
and $$R^3(abcdabc) =  \{abc, bcd, cda, dab, abc, bca, cab \}.$$
\begin{rem}
Equivalently, the $k$-range of $C$ is 
\begin{equation*}\label{R}R^k(C) = 
\SUB^k(C_0C_1\cdots C_{|C|+k-3}C_{|C|+k-2}) = \SUB^k(C \cdot C_0^{k-1}) .\end{equation*}
This can be thought of as the multiset of the $|C|$ different length-$k$ substrings of $C$ if we allow ``looping over'' from the end of $C$ to the beginning of $C$. 
\end{rem}

If $\mathcal{A}$ is a set of symbols, we say that $C$ is a \textbf{universal cycle} or a \textbf{ucycle} on $P_k(\mathcal{A})$ if $\Gamma\left(R^k(C)\right)  = P_k(\mathcal{A})$. In other words, if $C$ is a universal cycle on $P_k(\mathcal{A})$, then every string in the $k$-range of $C$ consists of $k$ different symbols in $\mathcal{A}$, the set of these $k$ symbols is different for each element of $C$'s $k$-range, and for any $k$ symbols in $\mathcal{A}$, there is some element of the $k$-range of $C$ which consists of these $k$ symbols.

Since  $P_k([n]) = \left[\begin{smallmatrix} n \\ k \end{smallmatrix} \right]$, this new definition agrees with the earlier definition of a universal cycle on $\left[\begin{smallmatrix} n \\ k \end{smallmatrix} \right]$.

\subsection{Rotations} 
We say that a \textbf{rotation} of a cycle $C$ is any cycle of the form $C_xC_{x+1}\cdots \allowbreak C_{x+|C|-1}$. In other words, a rotation of $C$ is any cycle which could be obtained from $C$ by repeatedly moving a symbol from the beginning of $C$ to the end of $C$. For example, the rotations of $``abcbc"$ are $``abcbc"$, $``bcbca"$, $``cbcab"$, $``bcabc"$, and $``cabcb"$.

Two simple but important facts follow from this definition. First, rotating a cycle does not change the $k$-range of the cycle for any $k$. Second, if $S$ is in the $k$-range of $C$, there exists some rotation $C'$ of $C$ such that $S = C'_0C'_1\cdots C'_{k-1}$; in other words, if $S$ is in $C$ $k$-range, we can always rotate $C$ so that it starts with $S$.

\section{Cycle Addition} \label{Cycle Addition}
In this section, we present a method for taking the $k$-sum ($\oplus^k$) of two cycles to get a new cycle. This operation requires the addends to have a common string of length at least $k-1$, and has several useful properties which we prove in Theorem 3.2 and its corollary.
\subsection{Construction}
Let $S$ be a $(k-1)$-string, and let $C$ and $D$ be cycles containing $S$. If $C'$ and $D'$ are rotations of $C$ and $D$ which both start with $S$, 
we say that $C' \cdot D'$ is a \textbf{k-sum} of $C$ and $D$. Note that $S$ here is arbitrary; all we require is that the first $k-1$ symbols of $C'$ and $D'$ match. 

If there is at least one cycle which is the $k$-sum of $C$ and $D$, we will use $C\oplus^k D$ to denote some (arbitrary) $k$-sum of $C$ and $D$. 

\subsection{Example}
For example, let $C = ``abc"$ and $D = ``bcdab"$. In this case the 3-sums of $C$ and $D$ are $abc \cdot abbcd = abcabbcd$ and $bca \cdot bcdab = bcabcdab$.

\subsection{Properties}
\begin{rem}
If $C$ and $D$ have intersecting $(k-1)$-ranges, then there is at least one $k$-sum of $C$ and $D$.
\end{rem}

\begin{lem}
If $E$ is a $k$-sum of $C$ and $D$, then $E$ is a $(k-1)$-sum of $C$ and $D$.
\end{lem}
\begin{proof} Since $E$ is a $k$-sum of $C$ and $D$, we can write $E = C' \cdot D'$, where $C'$ and $D'$ are rotations of $C$ and $D$ starting with the same $k-1$ characters. Since $C'$ and $D'$ start with the same $k-1$ characters, they must also start with the same $k-2$ characters, so $\, C' \cdot D' = E$ is also a $(k-1)$-sum of $C$ and $D$.
\end{proof}

\begin{theo}\label{sumsWork}
If $E$ is a $k$-sum of $C$ and $D$, then the $k$-range of $E$ is the disjoint-union of the $k$-ranges of $C$ and $D$. Formally, 
$$R^k (C \oplus^k D) = R^k (C) \cup R^k (D).$$
\end{theo}

\begin{proof}  Since $E$ is a $k$-sum of $C$ and $D$, we can write $E = C' \cdot D'$, where $C'$ and $D'$ are rotations of $C$ and $D$ starting with the same $k-1$ characters. Let us call the $(k-1)$-string of those first characters $S$.
\begin{align*}
R^k(C) = R^k(C') &= \SUB^k(C' \cdot S), \textnormal{ and}
\\R^k(D) = R^k(D') &=  \SUB^k(D' \cdot S). \phantom{\textnormal{ and}}
\end{align*}
 Thus, $$R^k(C) \cup R^k(D) = \SUB^k(C' \cdot S) \cup \SUB^k(D' \cdot S).$$ 
But since the last $k-1$ characters of $\SUB^k(C' \cdot S)$ are the same as the first $k-1$ characters of $D' \cdot S$,  $$\SUB^k(C' \cdot S) \cup \SUB^k(D' \cdot S) = \SUB^k(C' \cdot D' \cdot S) = \SUB^k(E \cdot S) = R^k(E).$$
\end{proof}

A simple example of this theorem can be seen for $C = ``abc''$, $D = ``bcde''$, and $k = 3$. Here, a 3-sum of $C$ and $D$ is $bca \cdot bcde = bcabcde$, and
\begin{align*}R^3(bcabcde) &= \{bca, cab, abc, bcd, cde, deb, dec\} \\&= \{bca, cab, abc\} \cup \{bcd, cde, deb, dec\} \\&= \{abc, bca, cab\} \cup \{bcd, cde, deb, dec\} = R^3(C) \cup R^3(D).\end{align*}

\begin{cor}\label{sumsWork2}
If $E$ is a $k$-sum of $C$ and $D$, then the $(k-1)$-range of $E$ is the disjoint-union of the $(k-1)$-ranges of $C$ and $D$.
\end{cor}
\begin{proof}
 By Lemma 2, If $E$ is a $k$-sum of $C$ and $D$, it is also a $(k-1)$-sum of $C$ and $D$. Thus, a straightforward application of Theorem \ref{sumsWork} tells us that $R^{k-1}(E) = R^{k-1}(C) \cup R^{k-1}(D)$. 
\end{proof}

\subsection{Cycle Summation}
Sometimes, we will want to take $k$-sums of more than 2 elements. This leads us to define a generalization over cycle addition which we will call cycle summation. If $\mathscr{C}$ is a set of cycles, we will say that it is \textbf{k-summable} if there exists a valid order in which we can add up all the elements of $\mathscr{C}$. If $\mathscr{C}$ is $k$-summable, we will furthermore define a \textbf{k-summation} of $\mathscr{C}$, denoted $\bigoplus^k \mathscr{C}$, to any $k$-sum of the elements of $\mathscr{C}$ taken in some valid order. 

Since $R^{k-1}(C \oplus^k D) = R^{k-1}(C) \cup R^{k-1}(D)$, some $C\oplus^k (D\oplus^k E)$ exists if and only if  some $C\oplus^k D$ or some $C\oplus^k E$ exists. Thus, $\mathscr{C}$  is $k$-summable if and only if for any $C,D \in \mathscr{C}$ there exists a set of cycles $C_0,C_1,\cdots,C_{n}$ in $\mathscr{C}$ such that $C = C_0$, $D = C_{n}$, and for any $i \in [n-1]$, some $C_i \oplus^k C_{i+1}$ exists.

\begin{rem}
We can extend the results of Theorem \ref{sumsWork} and Corollary \ref{sumsWork2} to $k$-summations. In other words, for any set of cycles $\mathscr{C}$, 
$$R^k\Big(\osum{k} \mathscr{C}\Big) = \bigcup_{C \in \mathscr{C}}R^k(C) \textnormal{ and}$$
$$R^{k-1}\Big(\osum{k} \mathscr{C}\Big) = \bigcup_{C \in \mathscr{C}}R^{k-1}(C).$$
\end{rem}

\section{Cycle Multiplication} \label{Cycle Multiplication}
In the previous section, we saw that if $C$ and $D$ were cycles satisfying certain simple conditions, we could find a cycle $C \oplus ^k D$ such that $R^k(C) \cup R^k (D) = R^k(C \oplus ^k D) $. It would be desirable to have an analogous result where we could find a cycle  $E$ such that $\Gamma \left(R^t(C)\right) \times \Gamma \left(R^u (D)\right) = \Gamma\left( R^{t+u}(E) \right)$. Unfortunately, such a cycle is sometimes impossible to find.\footnote{A simple example of this occurs when $t = u = 1$, $C=aaa$, and $D = b$. Then, $\Gamma \left(R^t(C)\right) \times \Gamma \left(R^u (D)\right)$ will be $\{\{a,b\}, \{a,b\}, \{a,b\}\}$, which cannot be the 2-range of any cycle.} Instead, we will show a slightly weaker result: as long as $|C|$ and $|D|$ are both multiples of $t+u$, there is a set of cycles $ \mathscr{C}$ such that $$\Gamma \left(R^t(C)\right) \times \Gamma \left(R^u (D)\right) = \bigcup_{E \in  \mathscr{C}} \Gamma\left( R^{t+u}(E) \right).$$
The elements of $ \mathscr{C}$ in this result are exactly the WEAVEs that we examine throughout this section.

\subsection {Construction}
Fix two positive integers $t$ and $u$, and let $k = t+ u$. Furthermore, let $C$ and $D$ be cycles such that both $|C|$ and $|D|$ are multiples of $k$. Then, for any integers $c$ and $d$, we will define $\WEAVE_{c,d}(C^t,D^u)$ to be the cycle $$C_c^t \cdot D_d^u \cdot C_{c+t}^t \cdot D_{d+u}^u \cdots C_{c+(r-1)t}^t \cdot D_{d+(r-1)u}^u$$ 
where $r = \frac{\lcm(|C|u, |D|t)}{t u}$. 

\begin{rem}
Since $k$ is a factor of both $|C|$ and $|D|$, $ k\cdot \lcm(t,u)=\lcm(ku, kt)$ is a factor of $\lcm(|C|u, |D|t)$. But $k = t+u$ is a multiple of $\gcd(t,u)$, so $ tu = \gcd(t,u) \cdot \lcm (t,u)$ divides $\lcm(|C|u, |D|t)$. Thus, $r$ is in fact an integer.
\end{rem}

Notice that we obtain $\WEAVE_{c,d}(C^t,D^u)$ by ``interweaving'' $C$ and $D$: we take $t$ characters from $C$, then $u$ characters from $D$, then $t$ characters from $C$, then $u$ characters from $D$, and so on. We continue this process, possibly looping over the cycles multiple times, until we simultaneously return to the place we started in both $C$ and $D$. Since $r = \frac{\lcm(|C|u, |D|t)}{tu}$ is the first value for which both $\frac {rt}{|C|}$ and $\frac {ru}{|D|}$ are integers, this happens after we have used $rt$ characters from $C$ and $ru$ characters from $D$.

\subsection {Example}
For example, let $t = 3, u=2, C = 12345,$ and $D = abcde.$ Then,  $$\WEAVE_{0,0}(C^3,D^2) = 123 \cdot ab \cdot 451 \cdot cd \cdot 234 \cdot ea \cdot 512 \cdot bc \cdot 345 \cdot de.$$
\subsection {Properties}
\begin{note}
Throughout this section, we will let $n$ be any integer, $m$ be any integer satisfying $0 \le m < k$, and $W$ be $\WEAVE_{c,d}(C^t,D^u)$. 
\end{note}

\begin{rem}
$|W| =rt + ru = rk$.
\end{rem}

\begin{rem}
$W_{nk}^k = C_{c+nt}^t \cdot D_{d+nu}^u$; that is, the length-$k$ substring of $W$ starting at index $nk$ is exactly the concatenation of the length-$t$ substring of $C$ starting at index $c+nt$ with the length-$u$ substring of $D$ starting at index $d+nu$. Note that this holds for all integers $n$, including those greater than $r$. 
%since $rt$ is divisible by $|C|$ and $ru$ is divisible by $|D|$, this hold 
\end{rem}

\begin{rem}
We can derive an explicit form for the symbol found at a given index of $W$:
$$W_{nk + m} = \begin{cases}C_{c+nt+m} & 0 \le m < t \\ 
D_{d+nu+(m-t)} & t \le m < k.   \end{cases}$$
This allows us to also find an explicit form for the $k$-substring of $W$ starting from a certain index:
$$W_{nk + m}^k = \begin{cases} 
C_{c+nt+m}^{t-m} \;\;\;\;\;\; \cdot  D_{d+nu}^u \;\;\;\;\;\; \cdot C_{c+(n+1)t}^{m}& 0 \le m <  t \vspace{13pt} \\ 
D_{d+nu+(m-t)}^{u-(m-t)} \cdot C_{c+(n+1)t}^t \cdot D_{d+(n+1)u}^{(m-t)} & t\, \le m < k.  
\end{cases}$$
Although this form is not particularly elegant, this result allows us to derive a much more manageable formulation for $\Gamma \big( W_{nk + m}^k\big)$ which will be fundamental to our proof of the Product Theorem. 
\end{rem}

\begin{lem} \label{gammaIs}
$$\Gamma \big( \big( \WEAVE_{c,d}(C^t,D^u) \big)_{nk + m}^k\big) = \begin{cases} 
\Gamma \big( C_{c+nt+m}^t \big) \cup \Gamma \big( D_{d+nu}^u \big)& 0 \le m <  t  \vspace{13pt}\\ 
\Gamma \big( C_{c+(n+1)t}^t \big) \cup \Gamma \big( D_{d+nu+(m-t)}^u \big) & t\, \le m < k. 
\end{cases}$$

\end{lem}

\begin{proof}
In the notation of this section, $\big( \WEAVE_{c,d}(C^t,D^u) \big)_{nk + m}^k = W_{nk + m}^k$, and we can use the result above to compute  

\begin{align*}
\Gamma \big( W_{nk + m}^k\big) 
& = \begin{cases} 
\Gamma \Big( C_{c+nt+m}^{t-m} \;\;\;\;\;\; \cdot  D_{d+nu}^u \;\;\;\;\;\; \cdot C_{c+(n+1)t}^{m} \Big)& 0 \le m < \, t \vspace{11pt} \\ 
\Gamma \Big( D_{d+nu+(m-t)}^{u-(m-t)} \cdot C_{c+(n+1)t}^t \cdot D_{d+(n+1)u}^{(m-t)} \Big) & t\, \le m < k 
\end{cases}
\\&= \begin{cases} 
\Gamma \Big( C_{c+nt+m}^t \cdot  D_{d+nu}^u \Big)& 0 \le m < \, t \vspace{11pt} \\
\Gamma \Big( D_{d+nu+(m-t)}^u \cdot C_{c+(n+1)t}^t  \Big) & t\, \le m < k  
\end{cases}
\\&= \begin{cases} 
\Gamma \big( C_{c+nt+m}^t \big) \cup \Gamma \big( D_{d+nu}^u \big)& 0 \le m < \, t \vspace{11pt} \\ 
\Gamma \big( C_{c+(n+1)t}^t \big) \cup \Gamma \big( D_{d+nu+(m-t)}^u \big) & t\, \le m < k. 
\end{cases}
\end{align*}
\end{proof}

\subsection{The Product Theorem}
We would like to prove
\begin{named}[Product Theorem]
Let $C$ and $D$ be any cycles for which $|C|$ and $|D|$ are both multiples of $t+u$. Then, there exists a value $s$ such that 
$$\Gamma \left(R^t(C)\right) \times \Gamma \left(R^u (D)\right) = \bigcup_{a=0}^{s-1} \Gamma\left(R^k\big( \WEAVE_{a,-a}(C^t,D^u) \big)\right).$$
\end{named}

We will start by defining two integer functions:
$$ F(nk+m)= \begin{cases} 
nt+m & 0 \le m < t \\ 
(n+1)t& t\, \le m < k 
\end{cases}$$
and
$$
G(nk+m) = \begin{cases}
nu & 0 \le m <  t \\ 
nu + (m-t)& t\, \le m < k. 
\end{cases}
$$
This allows us to write the result from Lemma \ref{gammaIs} in a simpler form:
$$\Gamma \Big( \big( \WEAVE_{c,d}(C^t,D^u) \big)_{nk + m}^k\Big) = \Gamma\big(C_{c+F(nk+m)}^t\big) \cup \Gamma\big(D_{d+G(nk+m)}^u\big), $$
Equivalently, if we substitute $i$ for $nk+m$, 
$$\Gamma \Big( \big( \WEAVE_{c,d}(C^t,D^u) \big)_{i}^k\Big) = \Gamma\big(C_{c+F(i)}^t\big) \cup \Gamma\big(D_{d+G(i)}^u\big), $$

If we let $H$ be the set $\{(F(i),G(i))  : 0 \le i < rk \}$, it follows that 
\begin{align*}
\Gamma \Big(R^k\big( \WEAVE_{c,d}(C^t,D^u) \big)\Big) &= \Big\{\Gamma\big(C_{c+F(i)}^t\big) \cup \Gamma\big(D_{d+G(i)}^u\big) : 0 \le i < rk\Big\}
\\& =  \Big\{\Gamma\big(C^t_{f+c}\big) \cup \Gamma\big(D^u_{g+d}\big) : (f,g) \in H\Big\}.
\end{align*}

To proceed beyond this point we will first need to prove some properties of $H$.

\begin{rem}
$F(i+k) = F(i) + t$ \ and \  $G(i+k) = G(i) + u$.
\end{rem}

\begin{lem}\label{canAdd}
$F(i) + G(i) = i$.
\end{lem}
\begin{proof}
 If we write $i = nk + m$, 
\begin{align*}
F(i) + G(i) 
&=F(nk+m) + G(nk+m)
\\&= \left\{\begin{array}{lr} 
\big(nt+m\big) + \big(nu\big) & 0 \le m < \, t \\ 
\big((n+1)t\big) + \big(nu+ (m-t)\big)& t \le m < k 
\end{array}\right.
\\&= \left\{\begin{array}{lr} 
nt + m + nu & 0 \le m < \, t \\ 
nt + t+nu +m - t& t \le m < k 
\end{array}\right.
\\&= n(t+u) + m
\\& = i.
\end{align*}
\end{proof}

Throughout this subsection, we will say that two ordered pairs of integers are \textbf{similar} ($\sim$) if their first coordinates are equivalent modulo $|C|$ and their second coordinates are equivalent modulo $|D|$. In other words, $(x_1,y_1) \sim (x_2,y_2)$ if and only if $x_1  \equiv x_2 \pmod {|C|}$ and $y_1  \equiv y_2 \pmod {|D|}$.

\begin{rem}
If $(x_1,y_1) \sim (x_2,y_2)$, then $C^t_{x_1} = C^t_{x_2}$, $D^u_{y_1} = D^u_{y_2}$, and consequently, 
$$\Gamma\big(C^t_{x_1}\big) \cup \Gamma\big(D^u_{y_1}) = \Gamma\big(C^t_{x_2}\big) \cup \Gamma\big(D^u_{y_2}).$$
 \end{rem}

\begin{lem}  \label{noRepeats} 
If $i$ and $j$ are integers, we will have $(F(i),G(i))\sim (F(j),G(j))$ if and only if $j-i$ is a multiple of $rk$.
\end{lem}
\begin{proof}

$\boldsymbol{\big(}\!\!  \boldsymbol\Leftarrow \hspace{-3pt}  \boldsymbol{\big)}\!\!:$  Let $j-i = ark$ for some integer $a$. 
Since $rt$ is a multiple of $|C|$ and $ru$ is a multiple of $|D|$,
\begin{gather*}
F(j) = F(i+ark) = F(i) + art \equiv F(i) \pmod{|C|}
\\G(j) = G(i+ark) = G(i) + aru \equiv G(i) \pmod{|D|}.
\end{gather*} 

$\boldsymbol{\big(}\!\!  \boldsymbol\Rightarrow \!\!  \boldsymbol{\big)}\!\!:$ Let us assume $(F(i),G(i))\sim (F(j),G(j))$. Since $|C|$ and $|D|$ are both multiples of $k$, $F(i) \equiv F(j) \pmod k$ and $G(i) \equiv G(j) \pmod k$, so Lemma \ref{canAdd} tells us that
$$
i = F(i) + G(i) \equiv F(j) +G(j)= j \pmod {k}.
$$
Thus, we can write $j = i + nk$ for some integer $n$. But $F(i+nk) = F(i) + nt$ and $G(i+nk) = G(i) + nu$,
so $n$ must satisfy both $nt \equiv 0 \pmod{|C|}$ and $nu \equiv 0 \pmod{|D|}.$ The only such values of $n$ are multiples of  $\frac{\lcm(|C|u, |D|t)}{t u} = r$, so $j-i = nk$ is divisible by $rk$.
\end{proof}

\begin{lem} \label{H loops}
For any $i$, exactly one $(f,g) \in H$ satisfies $(f,g)\sim (F(i),G(i))$.
\end{lem}

\begin{proof}
For any $i$, there is exactly one value $j \in [rk]$ satisfying $j \equiv i \pmod {rk}$. By Lemma \ref{noRepeats}, $j$ must be the only value in $[rk]$ satisfying $$(F(j),G(j)) \sim (F(i),G(i)),$$ so  $(f,g) = (F(j),G(j))$ is the only element of $H$ satisfying  $(f,g) \sim (F(i),G(i))$. 
\end{proof}

Let us define $s$ to be the smallest positive integer for which there exist $(f_1,g_1)$ and $(f_2,g_2)$ in $H$ satisfying $(f_2,g_2) \sim (f_1+s,g_1-s)$.\footnote{We know such an $s$ must exist because we can let $(f_1, g_1) = (f_2, g_2) =(F(0), G(0))$ and pick $s$ to be a multiple of both $|C|$ and $|D|$.}

\begin{lem}\label{noOverlaps}
If $(f_1,g_1)$ and $(f_2,g_2)$ are different elements of $H$ and $a$ and $b$ are integers such that $(f_1+a,g_1-a) \sim(f_2+b,g_2-b)$, then we must have $|a-b| \ge s$.
\end{lem}
\begin{proof}
By Lemma \ref{H loops}, we cannot have $(f_1,g_1) \sim (f_2,g_2)$, so our condition that $(f_1+a,g_1-a) \sim (f_2+b,g_2-b)$ implies $a \ne b$. Without loss of generality, let us assume that $a>b$.
\begin{gather*}
(f_2+b,g_2-b) \sim(f_1+a,g_1-a), \textnormal{ so}
\\(f_2,g_2) \sim(f_1+(a-b),g_1-(a-b)).
\end{gather*}
Since $(a-b)$ is a positive integer, by definition $s \le (a-b)$.
\end{proof}

\begin{lem}\label{canJump}
For any $i$ and any $x$, there must exist a $j$ satisfying $$(F(j),G(j)) \sim (F(i)+xs, G(i)-xs).$$
%
%Equivalently, for any $i^*$ and any $x$, there exists some $j^*$ such that $0 \le j^* < rk$, $F(j^*) \equiv F(i^*) - xs \pmod {|C|}$ and $G(j^*) \equiv G(i^*)+xs \pmod {|D|}$. 
\end{lem}
\begin{proof}
From the definition of $s$, we know there must exist some $i^*,j^*$ such that $(F(j^*),G(j^*)) \sim (F(i^*)+s,G(i^*)-s)$. By Lemma \ref{canAdd}, 
\begin{align*}
i^*-j^* &= \left(F(i^*) + G(i^*)\right) - \left(F(j^*) + G(j^*)\right)
\\&= \left(F(i^*) - F(j^*)\right) + \left(G(i^*) - G(j^*)\right)
\\& \equiv (s) + (-s) \pmod k.
\end{align*}
Therefore, $i^*-j^*$ is a multiple of $k$, so we can write $j^* = i^* + nk$, which allows us to compute
\begin{gather*}
\phantom{\textnormal{ so}}(F(i^*) + nt, G(i^*) + nu) = (F(j^*),G(j^*)) \sim (F(i^*)+s,G(i^*)-s),\textnormal{ so}
\\ (nt,nu) \sim (s,-s).
\end{gather*}

Let $x$ and $i$ be given, and let $j = i + xnk$. Then, 
\begin{align*}(F(j),G(j)) &= (F(i + xnk),G(i + xnk))  \\&= (F(i) + xnt, G(i) +xnu) \\&\sim (F(i) +xs, G(i)-xs).\end{align*}
\end{proof}

\begin{cor}\label{W loops}
For any $a$, $\WEAVE_{a,-a}(C^t,D^u)$ and $\WEAVE_{a+s,-a-s}(C^t,D^u)$ are rotations of each other. 

Recall that we have defined $s$ to be the smallest positive integer for which there exist $(f_1,g_1)$ and $(f_2,g_2)$ in $H$ satisfying $(f_2,g_2) \sim (f_1+s,g_1-s)$,
where
$$H = \{(F(i),G(i))  : 0 \le i < rk \},$$
where $F$ and $G$ given by  
\begin{align*}
 F(nk+m) &= \begin{cases} 
nt+m & 0 \le m < t \\ 
(n+1)t& t\, \le m < k 
\end{cases}
\\G(nk+m) &= \begin{cases}
nu & 0 \le m <  t \\ 
nu + (m-t)& t\, \le m < k. 
\end{cases}
\end{align*}
\end{cor}
\begin{proof}
As we saw in the proof of Lemma \ref{canJump}, there exists an $n$ which satisfies $(nt,nu) \sim (s,-s)$. Thus, for all $x$ and $y$, 
\begin{gather*}
F(x) + a+s \equiv F(x)+a+nt = F(x+nk) +a\pmod{|C|}
\\\,\,G(y) -a - s \equiv G(y)-a+nu = G(y+nk) -a\pmod{|D|},
\end{gather*}
so for any $i$, 
$$\big(\WEAVE_{a+s,-a-s}(C^t,D^u)\big)_i = \big(\WEAVE_{a,-a}(C^t,D^u)\big)_{i+nk}.$$
\end{proof}

\begin{comment}
\begin{cor}
{(f+x,g+y) : (f,g)\in H} and {(f+x+s,g+y-s) : (f,g)\in H} are equivalent up to (~)
\end{cor}
\begin{proof}
\end{proof}

\begin{cor}
There exists some $(f_1,g_1) \in H$ satisfying $f_1 \equiv f^* \pmod {|C|}$ and $g_1 \equiv g^* \pmod {|D|}$
if and only if there exists some $(f_2,g_2) \in H$ satisfying $f_2 \equiv f^* + s \pmod {|C|}$ and $g_1 \equiv g^* - s\pmod {|D|}$.
\end{cor}

\begin{proof}
By Lemma \ref{canJump}, if $(f_1,g_1) \in H$, then there exists $(f_2,g_2)\in H$ satisfying $f_2 \equiv f_1+s \pmod {|C|}$ and $g_2 \equiv g_1-s \pmod {|D|}$. Thus, if  $f_1 \equiv f^* \pmod {|C|}$ and $g_1 \equiv g^* \pmod {|D|}$, there  exists some $(f_2,g_2) \in H$ satisfying $f_2 \equiv f^* + s \pmod {|C|}$ and $g_1 \equiv g^* - s\pmod {|D|}$. The other direction is proved in the same way.
\end{proof}

\begin{proof}
Let $H+(x,y)$ denote $\big\{(f+x,g+y) : (f,g)\in H\big\}$

$$\WEAVE_{a,-a}(C^t,D^u) = \Big\{\Gamma\big(C_{f+a}\big) \cup \Gamma\big(D_{g-a}\big) : (f,g) \in H\Big\}, \textnormal{and}$$
$$\WEAVE_{a+s,-a-s}(C^t,D^u) = \Big\{\Gamma\big(C_{f+a+s}\big) \cup \Gamma\big(D_{g-a-s}\big) : (f,g) \in H\Big\},$$
so 
\end{proof}
\end{comment}

\begin{theo} \label{H works}
For any $x$ and $y$, there is a unique $(f,g) \in H$ and a unique $a \in [s]$ satisfying $(f+a,g-a) \sim (x,y)$.
\end{theo}
\begin{proof}
By Lemma \ref{noOverlaps}, for any $a$ and $b$ in $[s]$, there cannot be two different elements $(f_1,g_1), (f_2,g_2) \in H$ satisfying $(f_1 +a,g_1-a) \sim (f_2 +b,g_2-b)$. In addition, if $a,b\in[s]$ are distinct, $(f+a,g+a) \not \sim (f+b,g+b)$. 
Thus, if some $a\in [s]$ and $(f,g)\in H$ satisfy the conditions of this theorem, they do so uniquely.

Let $i =  x + y$, and let $\mu = G(i) - y$.  By Lemma \ref{canAdd} $F(i)+G(i) = i$, so 
$$F(i) +\mu = F(i) + G(i) - y = i - y = x.$$
Thus, $(F(i) +\mu, G(i)-\mu) \sim (x,y)$.

Let $a$ be the value satisfying $a\in[s]$ and $a \equiv \mu \pmod{s}$. By Lemma \ref{canJump}, there exists some $j$ for which $(F(j),G(j))\sim(F(i) + (\mu-a), G(i) - (\mu-a))$. Then,
$$(F(j)+a,G(j)-a) \sim (F(i)+\mu,G(i)-\mu)\sim (x,y) .$$

By Lemma \ref{H loops}, there exists $(f,g)\in H$ satisfying $(f,g) \sim (F(j),G(j))$, so\\  $(f+a,g-a) \sim (x,y)$.
\end{proof}

Let $H^*$ denote the set of ordered pairs $\{ (f+a,g-a) : (f,g)\in H, a \in [s] \}$, and let $J$ denote the set of ordered pairs $\{ (x,y) : x\in [|C|], y \in [|D|] \}$.
\begin{cor} \label{H bijection}
There is a bijection between $H^*$ and $J$ which maps ordered pairs to similar ordered pairs.
\end{cor}
\begin{proof}
Let $B:H^* \rightarrow J$ be a map which takes any ordered pair in $H^*$ to the element of $J$ which it is similar to. By  Theorem \ref{H works}, for any $ (x,y)\in J$, there is exactly one element $(f+a,g-a) \in H^*$ such that $B((f+a,g-a)) = (x,y)$, so $B$ must be a bijection.
\end{proof}

We can finally prove the Product Theorem.

\begin{named}[Product Theorem]
Let $C$ and $D$ be any cycles for which $|C|$ and $|D|$ are both multiples of $k=t+u$. 

Let $s$ be the smallest positive integer for which there exist $(f_1,g_1)$ and $(f_2,g_2)$ in $H$ satisfying $(f_2,g_2) \sim (f_1+s,g_1-s)$,
where
$$H = \{(F(i),G(i))  : 0 \le i < rk \},$$
where $F$ and $G$ given by  
\begin{align*}
 F(nk+m) &= \begin{cases} 
nt+m & 0 \le m < t \\ 
(n+1)t& t\, \le m < k 
\end{cases}
\\G(nk+m) &= \begin{cases}
nu & 0 \le m <  t \\ 
nu + (m-t)& t\, \le m < k. 
\end{cases}
\end{align*}

Then, 
$$\Gamma \left(R^t(C)\right) \times \Gamma \left(R^u (D)\right) = \bigcup_{a=0}^{s-1} \Gamma\left(R^{t+u}\big( \WEAVE_{a,-a}(C^t,D^u) \big)\right).$$
\end{named}
\begin{proof}
We know from the discussion preceding Lemma \ref{canAdd} that 
$$\Gamma \Big(R^k\big( \WEAVE_{c,d}(C^t,D^u) \big)\Big)
 =  \Big\{\Gamma\big(C^t_{f+c}\big) \cup \Gamma\big(D^u_{g+d}\big) : (f,g) \in H\Big\}.$$ It follows that
$$\bigcup_{a=0}^{s-1} \Gamma\left(R^{t+u}\big( \WEAVE_{a,-a}(C^t,D^u) \big)\right) =
\Big\{\Gamma\big(C^t_{f}\big) \cup \Gamma\big(D^u_{g}\big) : (f,g) \in H^*\Big \}.
$$
If $(f,g) \sim (x,y)$ then $\Gamma\big(C^t_{f}\big) \cup \Gamma\big(D^u_{g}\big) = \Gamma\big(C^t_{x}\big) \cup \Gamma\big(D^u_{y}\big)$, so by Corollary \ref{H bijection}, 
\begin{align*}
\Big\{\Gamma\big(C^t_{f}\big) \cup \Gamma\big(D^u_{g}\big) : (f,g) \in H^*\Big\}
&=\Big\{\Gamma\big(C^t_{x}\big) \cup \Gamma\big(D^u_{y}\big) : (x,y) \in J\Big\}
\\&= \Big\{\Gamma\big(C^t_{x}\big): x \in [|C|]\Big\} \times \Big\{\Gamma\big(D^u_{y}\big) : y\in[|D|]\Big\}
\\&= \Gamma\big(R^t(C)\big) \times \Gamma\big(R^u(D)\big).
\end{align*}
\end{proof}

\begin{rem}
An application of the Product Theorem shows that 
$$\left|\Gamma \left(R^t(C)\right) \times \Gamma \left(R^u (D)\right) \right|= \left|\bigcup_{a=0}^{s-1} \Gamma\left(R^{t+u}\big( \WEAVE_{a,-a}(C^t,D^u) \big)\right)\right|,$$
so $$|C| \cdot |D| = s \cdot \left|  \WEAVE_{a,-a}(C^t,D^u) \right| = srk = sk \frac{\lcm(|C|u, |D|t)}{t u}.$$
Therefore, we can explicitly compute $$ s = \frac{\gcd(|C|u, |D|t)}{k}.  $$
\end{rem}

\section{Benign Cycles}\label{pre-last}
The Product Theorem show that we can construct a class of cycles 
$$\mathscr{C} =  \big\{\WEAVE_{a,-a}(C^t,D^u): a \in [s] \big\}$$
with the property that $$\bigcup_{E \in \mathscr{C}}  \Gamma\Big(R^k(E)\Big) = \Gamma \left(R^t(C)\right) \times \Gamma \left(R^u (D)\right). $$
However, this is still of little use to us as long as $|\mathscr{C}|$ is large. In this section, we will introduce a method which will allow us to drastically reduce the cardinality of $|\mathscr{C}|$ when the cycle $|C|$ is $(t,t+u)$-benign. This will leave us with sufficiently few cycles so that we can eventually use cycle addition to construct our universal cycle. 

\subsection{Definition}
We say that a cycle $C$ is \textbf{(t,k)-benign} if for some $\Delta$ relatively prime to $|C|$ and some $i$, $C^{t-1}_i = C^{t-1}_{i+k\Delta}$. If $C$ is also a universal cycle on $S$, we would say that $C$ is a $(t,k)$-benign universal cycle on $S$.

\subsection{Examples}
For example, the cycle $C = abcdaeed$ is (3,4)-benign since $C^2_3 = da = C^2_7$ and $\frac{7-3}{4} = 1$ is an integer relatively prime to $|C| = 8$.

\subsection{Application}
As usual, let $k = t+u$.
\begin{lem} \label{benignA}
If $C$ and $D$ are cycles with lengths divisible by $k$ and $C$ satisfies $C^{t-1}_i = C^{t-1}_{i+k\Delta}$,  then for any $a$, we can find a $k$-sum $$\WEAVE_{a,-a}(C^t,D^u)\oplus^k\WEAVE_{a+u\Delta,-a-u\Delta}(C^t,D^u).$$
\end{lem}

\begin{proof}

Let $\Cbar$ denote the rotation of $C$ forward by $k \Delta$ spaces, so $\Cbar$ satisfies $\Cbar_x = C_{x+k\Delta}$ for all $x$. Notice that $\Cbar^{t-1}_i = C^{t-1}_{i+k\Delta} = C^{t-1}_i$.

For some $j$, $\big(\WEAVE_{c,d}(C^t,D^u)\big)^{k-1}_j$ consists of $C^{t-1}_i$ interspersed in some way with $u$ characters from $D$. But that means $\big(\WEAVE_{c,d}(\Cbar^t,D^u)\big)^{k-1}_j$ will consist of  $\Cbar^{t-1}_i$ interspersed in the same way with the same $u$ characters from $D$. Since  $\Cbar^{t-1}_i = C^{t-1}_i$, 
$$\big(\WEAVE_{c,d}(\Cbar^t,D^u)\big)^{k-1}_j = \big(\WEAVE_{c,d}(C^t,D^u)\big)^{k-1}_j.$$

We know from section \ref{Cycle Multiplication} that 
\begin{gather*}
\big(\WEAVE_{c,d}(C^t,D^u)\big)_{nk + m} = \begin{cases}C_{c+nt+m} & 0 \le m < t \\ 
D_{d+nu+(m-t)} & t \le m < k   \end{cases}, \textnormal{ so}
\\
\begin{aligned}
\big(\WEAVE_{c+u\Delta,d-u\Delta}(C^t,D^u)\big)_{nk + m}
&= \begin{cases}C_{c+u\Delta+nt+m} & 0 \le m < t \\ 
D_{d-u\Delta+nu+(m-t)} & t \le m < k   \end{cases}
\\&= \begin{cases}\Cbar_{c+(n-\Delta)t+m} & 0 \le m < t \\ 
D_{d+(n-\Delta)u+(m-t)} & t \le m < k   \end{cases}
\\&=\big(\WEAVE_{c,d}(\Cbar^t,D^u)\big)_{(n-\Delta)k + m}.
\end{aligned}
\end{gather*}
Therefore, we can conclude that 
\begin{align*}
\big(\WEAVE_{c+u\Delta,d-u\Delta}(C^t,D^u)\big)^{k-1}_{j+k\Delta} 
&=\big(\WEAVE_{c,d}(\Cbar^t,D^u)\big)^{k-1}_{j}
\\&=\big(\WEAVE_{c,d}(C^t,D^u)\big)^{k-1}_j,
\end{align*}
so for any $c$ and $d$, we can find a $k$-sum $$\WEAVE_{c,d}(C^t,D^u)\oplus^k\WEAVE_{c+u\Delta,d-u\Delta}(C^t,D^u).$$

By setting $d=-c$, this reduces to the result we were looking for.
\end{proof}

\begin{lem}\label{benignB}
If $C$ and $D$ are cycles with lengths divisible by $k$, $C$ is a $(t,k)$-benign cycle, and $ s = \frac{\gcd(|C|u, |D|t)}{k}  $, then there exists a partition of $$\mathscr{C} =  \big\{\WEAVE_{a,-a}(C^t,D^u): a \in [s] \big\}$$ into $\gcd(u,s)$ multisets $\mathscr{C}_i$ such that 
\begin{enumerate}
\item For any $a\in[s]$, if $i \in [\gcd(u,s)]$ and $a$ is equivalent to $i$ modulo $\gcd(u,s)$,  $\WEAVE_{a,-a}(C^t,D^u) \in \mathscr{C}_i$, and
\item Each $\mathscr{C}_i$ is $k$-summable.
\end{enumerate}
\end{lem}

\begin{proof}
Let $W_a$ denote $\WEAVE_{a,-a}(C^t,D^u)$.

Since $C$ is a $(t,k)$-benign cycle, we can find $\Delta$ relatively prime to $|C|$ and $i$ such that $C^{t-1}_i = C^{t-1}_{i+k\Delta}$. 
By Lemma \ref{benignA}, for any $a$ there exists a $k$-sum $W_a\oplus^k W_{a+u\Delta}$.
By Corollary \ref{W loops},  $W_a$ and $W_{a+s}$ are equivalent up to rotation for any $a$, so if  $a$ and $b$ satisfy the relation $b \equiv a +u\Delta \pmod {s}$, we can take the $k$ sum of $W_{a}$ and $W_{b}$.

Since $\Delta$ is relatively prime to $|C|$ and $s$ divides $|C|$, $\Delta$ must be relatively prime to $s$.  Thus, there must exist a value $\Deltabar$ satisfying $\Delta \Deltabar \equiv 1 \pmod {s}$.

Let $\mathscr{C}_i$ be the multiset of $W_a$ for which $a- i$ is a multiple of $\gcd(u,s)$. Notice that $\{ \mathscr{C}_1,\mathscr{C}_2,\cdots,\mathscr{C}_{\gcd(u,s)} \}$ is a partition of $\mathscr{C}$, and $W_a \in C_i$ for any $a$ equivalent to $i$ modulo $\gcd(u,s)$.

For any $W_a, W_b \in \mathscr{C}_i$, $a$ is equivalent to $b$ modulo  $\gcd(u,s)$. Since any multiple of $\gcd(u,s)$ can be written as an integer linear combination of $u$ and $s$, and $b-a$ is a multiple of $\gcd(u,s)$, there must exist integers $y,z$ such that $b-a = yu + zs$. Therefore, 
 $$b = a + y u + z s \equiv a + y u \equiv a + (y\Deltabar)(u\Delta) \pmod{s}.$$ 
If we let $a_x = a + x u\Delta$,  we get $a_0 = a$, $a_{y  \raisebox{-1pt}{$\scriptstyle {\hspace{1pt}\overline{\hspace{-1pt}\Delta\hspace{-1pt}}\hspace{1pt}}$}   } = b$, and $a_{i+1} \equiv a_{i}+u\Delta  \pmod {s}$. By the last condition, we can take a $k$-sum of $W_{a_i}$ and $W_{a_{i+1}} = W_{(a_{i} + u\Delta)}$, so by the criterion established in Section \ref{Cycle Addition},  $\mathscr{C}_i$ must be $k$-summable.
\end{proof}

\begin{cor}\label{benignC}
If $C$ and $D$ are cycles having lengths divisible by $k$ and $C$ is $(t,t+u)$-benign, then there exist $x \le u$ multisets $\mathscr{C}_i$ such that
\begin{enumerate}
\item For any $a$, if $i \in [x]$ and $i \equiv a \pmod x$,  $\WEAVE_{a,-a}(C^t,D^u) \in \mathscr{C}_i$, 
\item Each $\mathscr{C}_i$ is $k$-summable, and
\item If we let $\mathscr{C}$ denote $\bigcup_{i=0}^{x-1}\mathscr{C}_i $,
$$\Gamma \big(R^t(C)\big) \times \Gamma \big(R^{u} (D)\big) = \bigcup_{E \in  \mathscr{C} }\Gamma \big( R^{t+u}(E) \big).$$

\end{enumerate}
\end{cor}
\begin{proof}
Let $ s = \frac{gcd(|C|u, |D|t)}{k}  $ and $x = \gcd(u,s)$ (note that $\gcd(u,s) \le u$, so $x$ satisfies $x \le u$). By the product theorem, 
$\mathscr{C} =  \big\{\WEAVE_{a,-a}(C^t,D^u): a \in [s] \big\}$ satisfies  $\Gamma \big(R^t(C)\big) \times \Gamma \big(R^{u} (D)\big)=\bigcup_{E \in  \mathscr{C}}\Gamma \left( R^{t+u}(E) \right)$, so this Corollary follows directly from Lemma \ref{benignB}.
\end{proof}

\subsection{Existence of Important Cases}
\begin{rem}
Since $C_i^0 = C_j^0$ for any $i,j$,  any cycle is (1,k)-benign for arbitrary $k$.
\end{rem}
\begin{lem} \label{2benign}
For any $k>3$ and any odd $n \ge 2k-1$, there exists a (2,k)-benign universal cycle on $\left[ \begin{smallmatrix} n \\ 2 \end{smallmatrix} \right]$.
\end{lem}
\begin{proof}
For any $w$, let $D_w(x)$ be the cycle which has length $\frac{n}{\gcd(w,n)}$ whose symbols are given by $\big(D_w(x)\big)_i \equiv x + iw \pmod{n}, \big(D_w(x)\big)_i \in [n]$.\footnote{Note that our definition of $\big(D_w(x)\big)_i \equiv x + iw \pmod{n}$ is periodic, with a period exactly equal to the length of $D_w(x)$.}
 Less formally, $D_w(x)$ is the unique cycle which starts at $x$, has symbols taken from $[n]$, obeys the condition that each symbol must be $w$ greater (modulo $n$) than the last, and goes until it loops back to $x$ for the first time. 

Note that all of the $\frac{n}{\gcd(w,n)}$ symbols of $D_w(x)$ are unique, and are actually the symbols in $[n]$ which are equivalent to $x$ modulo $\gcd(w,n)$. Thus, the 2-range of $D_w(x)$ will be the set of strings $ij$ for which $i$ and $j$ are both in $[n]$, $i$ is equivalent to $x$ modulo $\gcd(w,n)$, and $j \equiv i +w \pmod n$.

 Now, let $\mathscr{D}_w = \{ D_w(x) : 0\le x < \gcd(w,n)\}$. We can see that $$\bigcup_{D \in \mathscr{D}_w} R^2(D) = \Big\{ij: i\in [n], j\in [n], j \equiv i + w \pmod {n}  \Big\}.$$
If we also let  $\mathscr{D} = \bigcup_{w=1}^{\frac{n-1}{2}} \mathscr{D}_w,$ then 
$$\bigcup_{D \in \mathscr{D}} \Gamma \left( R^2(D) \right) = \Big\{\{i,j\}: i\in [n], j\in [n], i \ne j \Big\} = P_2\big([n]\big).$$
Let  $\mathscr{D}' = \mathscr{D} - \{ D_{k-1}(0), D_1(0) \}$. Since $\frac{n-1}2 \ge k > 3$, and $k \ne 2$, both $\mathscr{D}$ and $\mathscr{D}'$ will contain  $D_2(0)$. 

Since 2 must be relatively prime to $n$, $D_2(0)$ must contain every symbol in $[n]$, which means its 1-range must intersect with the 1-range of every element of $\mathscr{D}'$. Thus, $\mathscr{D}'$ is 2-summable. \\Let $E$ denote some such 2-summation with a first symbol of `0'.

Since $D_{n-1}(0)^1_0 = D_{k-1}(0)^1_0 = E_0^1 = 0$, we can take their $k$-summation $E' = D_{n-1}(0) \cdot D_{k-1}(0) \cdot E$. But $$\Gamma\Big( R^2 \big( D_{n-1}(0) \big)  \Big) =\Gamma\Big( R^2 \big( D_1(0) \big)  \Big),$$ so 
\begin{align*}
\Gamma\Big( R^2(E')  \Big) 
&= \Gamma\Big( R^2\big( D_{n-1}(0) \big)  \Big) \cup \Gamma\Big( R^2\big( D_{k-1}(0) \big)  \Big) \cup \Gamma\Big( R^2\big(E\big)  \Big)
\\&= \Gamma\Big( R^2\big( D_1(0) \big)  \Big) \cup \Gamma\Big( R^2\big( D_{k-1}(0) \big)  \Big) \cup \Gamma\Big( R^2\big(E\big)  \Big)
\\& = \Gamma \Big(    \bigcup_{D \in \mathscr{D}} \left( R^2(D) \right)    \Big)
\\& = \Big\{\{i,j\}: i\in [n], j\in [n], i \ne j \Big\}.
\end{align*}
Thus, $E'$ is a universal cycle on $\left[ \begin{smallmatrix} n \\ 2 \end{smallmatrix} \right]$.

$E'_{n-(k-1)} = D_{n-1}(0)_{n-(k-1)} = k-1$ and $E'_{n+1} =  D_{k-1}(0)_1 = k-1$. Since $n+1 - (n - (k-1)) = k \cdot 1$ and $1$ is relatively prime to $|E'|$,  $E'$ must be $(2,k)$-benign.
\end{proof}

\section{Proof of the Main Theorem} \label{Main}
In this section, we will finally prove our main theorem:
\begin{named}[Main Theorem]
If $a$ and $b$ are positive multiples of 8 such that  neither $a+1$ nor $b+1$ are divisible by 3, then if there exist  universal cycles on $\left[\begin{smallmatrix} a+2 \\ 4 \end{smallmatrix} \right]$ and $\left[\begin{smallmatrix} b+2 \\ 4 \end{smallmatrix} \right]$, there must exist  universal cycles on $\left[\begin{smallmatrix} a+b+2 \\ 4 \end{smallmatrix} \right]$.
\end{named}

\subsection{Preliminaries}
\begin{lem}\label{Ucycle substitution}
If $C$ is a universal cycle on  $P_k(\mathcal{A})$, $|\mathcal{A}|=|\mathcal{B}|$, $S$ is a $(k+1)$-string consisting of $k+1$ distinct symbols from $\mathcal{B}$, and $x$ is any integer, 
then there exist a cycle $D$ such that
\begin{enumerate}
\item $D^{k+1}_x = S$
\item $D$ is a universal cycle on $P_k(\mathcal{B})$, and
\item if $C$ is $(a,b)$-benign, then so is $D$.
\end{enumerate}
\end{lem}
\begin{proof}
Since $|\mathcal{A}| =|\mathcal{B}|$, we can find a bijection from $\mathcal{A}$ to $\mathcal{B}$. Furthermore, for any distinct $a_1,a_2,\cdots,a_n \in \mathcal{A}$ and distinct $b_1,b_2,\cdots,b_n \in \mathcal{B}$, we can find such a bijection which maps each $a_i$ to $b_i$.

Since $C$ is universal cycle on  $P_k(\mathcal{A})$, $C^k_x$ and $C^k_{x+1}$ must each consist of $k$ different symbols. In addition, since we must have $\Gamma\big( C^k_x \big) \ne \Gamma\big( C^k_{x+1} \big)$, $C_{x} \ne C_{x+k}$, so $C^{k+1}_x$ consists of $k+1$ different characters.

Let $f$ be some bijection from $\mathcal{A}$ to $\mathcal{B}$ which takes $C_{x+i}$ to $S_i$ for every $i \in [k+1]$, and let $D = f(C)$. Then, for $i \in [k+1]$, $D_{x+i} = f(C_{x+i}) = S_i$, so $D^{k+1}_{x} = S$.

In addition, 
\begin{align*}
R^k(D) &= \{D_x^k : 0 \le x \le |D| - 1 \} 
\\&= \{f(C)_x^k : 0 \le x \le |C| - 1 \}
\\&= \{f(C_x^k) : 0 \le x \le |C| - 1 \}
\\&= f(\{C_x^k : 0 \le x \le |C| - 1 \})
\\&= f(P_k(\mathcal{A}))
\\&= P_k(\mathcal{B}).
\end{align*}

Finally, if $C$ is $(a,b)$-benign, then $C^{a-1}_i=C^{a-1}_{i+b\Delta}$, so $$D^{a-1}_i=f(C)^{a-1}_i = f(C^{a-1}_i) =  f(C^{a-1}_{i+b\Delta}) =  f(C)^{a-1}_{i+b\Delta} = D^{a-1}_{i+b\Delta},$$
which shows that $D$ must also be $(a,b)$-benign.
\end{proof}

Let $a$ and $b$ be positive integers for which
\begin{enumerate}
\item Both $a$ and $b$ are divisible by 8,
\item neither $a+1$ nor $b+1$ are divisible by 3, and
\item there exist universal cycles on $\left[\begin{smallmatrix} a+2 \\ 4 \end{smallmatrix} \right]$ and $\left[\begin{smallmatrix} b+2 \\ 4 \end{smallmatrix} \right]$.
\end{enumerate}
\begin{rem}
$a \ge 16$ and $b \ge 16$.
\end{rem}

Let $\mathcal{A}$ and $\mathcal{B}$ be disjoint sets of symbols satisfying $|\mathcal{A}| = a$ and $|\mathcal{B}| = b$. Let $\alpha$ and $\beta$ be distinct symbols not in $\mathcal{A} \cup \mathcal{B}$. 

Let 
\begin{alignat*}{2}
M_0 &= P_4 \left(\mathcal{A} \cup \{\alpha, \beta\} \right)
\\M_1 &= P_3 \left(\mathcal{A} \cup \{\alpha\} \right)& \times &  P_1 \left(\mathcal{B} \right)
\\M_2 &= P_2 \left(\mathcal{A} \cup \{\alpha\} \right)& \times  & P_2 \left(\mathcal{B} \cup \{\beta\} \right)
\\M_3 &=  P_1 \left(\mathcal{A} \right)& \times &  P_3 \left(\mathcal{B} \cup \{\beta\} \right)
\\M_4 &=  & & P_4 \left(\mathcal{B} \cup \{\alpha, \beta\} \right).
\end{alignat*}
\begin{rem}
$$M_0 \cup M_1 \cup M_2 \cup M_3 \cup M_4 = P_4 \left(\mathcal{A} \cup \mathcal{B} \cup \{\alpha, \beta\} \right).$$
\end{rem}

\subsection{Constructing the Component Cycles}
\begin{note}
The properties of cycles constructed in this subsection are summarized in Figure 1.
\end{note}
Since $a+1$ is odd and greater than $7=2\cdot4-1$, by Lemma \ref{2benign} there exists a $(2,4)$-benign universal cycle on $\left[ \begin{smallmatrix} a+1 \\ 2 \end{smallmatrix} \right]$. Thus, by Lemma \ref{Ucycle substitution}, we can find a cycle $C(2)$ which is a $(2,4)$-benign ucycle on $P_2\big(\mathcal{A} \cup \{\alpha\}\big) $ satisfying $\alpha \not \in \Gamma \left( C(2)^3_{1} \right)$. \\By similar reasoning, we can find a ucycle on $P_2\big(\mathcal{B} \cup \{\beta\}\big) $ satisfying $\beta \not \in \Gamma \left( D(2)^3_{-2} \right)$.

For any set $M$, we can obtain a universal cycle on $P_1\big(M\big) $ simply by listing the characters of $M$ in any order. Since $C(2)^3_1$ contains only characters from $\mathcal{A}$, we can find a cycle $C(3)$ which is a universal cycle on $P_1\big(\mathcal{A}\big) $ satisfying $C(3)^3_0 = C(2)_3 C(2)_1 C(2)_2$. 
\\By similar reasoning, we can find a cycle $D(1)$ which is a universal cycle on $P_1\big(\mathcal{B}\big) $ satisfying $D(1)^3_{0} = D(2)_{0} D(2)_{-2} D(2)_{-1}$.

Since $a+1\ge8$ and $a+1$ is not a multiple of 3, by the results of Jackson \cite{J}, there exist universal cycles on $\left[ \begin{smallmatrix} a+1 \\ 3 \end{smallmatrix} \right]$. Thus,  by Lemma \ref{Ucycle substitution}, we can find a cycle $C(1)$ which is a universal cycle on $P_3\big(\mathcal{A}\cup \{\alpha\}\big) $ satisfying $C(1)^4_{-2} = C(2)^4_{0}$, and therefore also satisfying  $C(1)^3_{-2} = C(2)^3_{0}$.
\\By similar reasoning, we can find a cycle $D(3)$ which is a universal cycle on $P_3\big(\mathcal{B}\cup \{\beta\}\big) $ satisfying $D(3)^3_{-2} = D(2)^3_{-1}$.

By assumption, there exist universal cycles on $\left[\begin{smallmatrix} a+2 \\ 4 \end{smallmatrix} \right]$ and $\left[\begin{smallmatrix} b+2 \\ 4 \end{smallmatrix} \right]$. Thus, by Lemma \ref{Ucycle substitution}, we can find a cycle $C(0)$ which is  a universal cycle on $P_4\big(\mathcal{A}\cup \{\alpha,\beta\}\big) $ satisfying $C(0)^4_{-1} = C(1)^4_{-1}$.
\\By similar reasoning, we can find a cycle  $D(4)$ which is  a universal cycle on $P_4\big(\mathcal{B}\cup \{\alpha,\beta\}\big) $ satisfying $D(4)^4_0 = D(3)^4_{-1}$.

\begin{rem}
$|C(1)|,|C(2)|,|C(3)|,|D(1)|,|D(2)|,$ and $|D(3)|$ are each divisible by 4.
\end{rem}

\begin{rem}
$$\begin{alignedat}{2}
M_0 &= \Gamma \left( R^4 \left(   C(0)   \right) \right)
\\M_1 &= \Gamma \left( R^3 \left(   C(1)   \right) \right)& \times & \Gamma \left( R^1 \left(   D(1)   \right) \right)
\\M_2 &= \Gamma \left( R^2 \left(   C(2)   \right) \right)& \times &\Gamma \left( R^2 \left(   D(2)   \right) \right)
\\M_3 &= \Gamma \left( R^1 \left(   C(3)   \right) \right)& \times &\Gamma \left( R^3 \left(   D(3)   \right) \right)
\\M_4 &=  & & \Gamma \left( R^4 \left(   D(4)   \right) \right).
\end{alignedat}
$$
\end{rem}

\begin{figure}[t]
\begin{align}
&C(3)^3_0 = C(2)_3 C(2)_1 C(2)_2,\textnormal{ so}  & C(3)_{0} &=  C(2)_3
\\ &C(3)^3_0 = C(2)_3 C(2)_1 C(2)_2,\textnormal{ so}  & C(3)_{1} &=  C(2)_1
\\ &C(3)^3_0 = C(2)_3 C(2)_1 C(2)_2,\textnormal{ so}  & C(3)_{2} &=  C(2)_2
\\\notag
\\ &C(1)^3_{-2} =  C(2)^3_{0},\textnormal{ so}  & C(1)^2_{-2} &=  C(2)^2_{0}
\\ &C(1)^3_{-2} =  C(2)^3_{0},\textnormal{ so}  & C(1)^2_{-1} &=  C(2)^2_{1}
\\\notag
\\&C(0)^4_{-1} = C(1)^4_{-1},\textnormal{ so} & C(0)^3_{-1} &= C(1)^3_{-1}
\\&C(0)^4_{-1} = C(1)^4_{-1},\textnormal{ so} & C(0)^3_0 &= C(1)^3_0
\\\notag
\\&D(1)^3_{0} = D(2)_{0} D(2)_{-2} D(2)_{-1},\textnormal{ so}  & D(1)_{0} &= D(2)_0
\\&D(1)^3_{0} = D(2)_{0} D(2)_{-2} D(2)_{-1},\textnormal{ so}  & D(1)_{1} &= D(2)_{-2}
\\&D(1)^3_{0} = D(2)_{0} D(2)_{-2} D(2)_{-1},\textnormal{ so}  & D(1)_{2} &= D(2)_{-1}
\\\notag
\\& D(3)^3_{-2} = D(2)^3_{-1},\textnormal{ so}  & D(3)^2_{-2} &= D(2)^2_{-1}
\\& D(3)^3_{-2} = D(2)^3_{-1},\textnormal{ so}  & D(3)^2_{-1} &= D(2)^2_{0}
\\\notag
\\&D(4)^4_0 = D(3)^4_{-1},\textnormal{ so} & D(4)^3_0 &= D(3)^3_{-1}
\\&D(4)^4_0 = D(3)^4_{-1},\textnormal{ so} & D(4)^3_1 &= D(3)^3_{0}
\end{align}
\vspace{-24pt}\caption{Summary of what we know by construction}

\end{figure}

\subsection{Fitting Everything Together}
Let us define 
\begin{gather*}
\begin{aligned}
E_i(1) = \WEAVE_{i,-i}\big(D(1)^1,C(1)^3\big),
\\E_i(2)= \WEAVE_{i,-i}\big(C(2)^2,D(2)^2\big),
\\E_i(3)= \WEAVE_{i,-i}\big(C(3)^1,D(3)^3\big),
\end{aligned}
\\
\begin{aligned}
\mathscr{H}(1) &= \{E_0(1), E_1(1), E_2(1)\},
\\\mathscr{H}(2) &= \{E_0(2), E_1(2)\} \textnormal{, and }
\\\mathscr{H}(3) &= \{E_0(3), E_1(3), E_2(3)\}.
\end{aligned}
\end{gather*}

\begin{lem}\label{itFits}
$\{C(0), D(4)\} \cup \mathscr{H}(1) \cup \mathscr{H}(2) \cup \mathscr{H}(3) $ is 4-summable.
\end{lem}
\begin{proof}
First,
\begin{align*}
C(0)^3_{-1} &= C(1)^3_{-1} = E_1(1)^3_1 & \textnormal{By } Fig.1(6)
\\C(0)^3_0 &= C(1)^3_0 = E_0(1)^3_1 & \textnormal{By } Fig.1(7).
\end{align*}
Thus, $\{ C(0), E_0(1), E_1(1)  \}$ must be 4-summable.
\begin{align*}
&E_0(1)^3_{-2} = C(1)^2_{-2} D(1)_{0} = C(2)^2_0 D(2)_0 =  E_0(2)^3_{0} & \textnormal{By } Fig.1(4,8)
\\&E_2(1)^3_{0} = D(1)_{2}C(1)^2_{-2}  = D(2)_{-1}C(2)^2_0 =  E_0(2)^3_{-1}& \textnormal{By } Fig.1(4,10),
\end{align*}
so $\{ C(0),E_0(2)\} \cup \mathscr{H}(1)$ must be 4-summable.
\begin{align*}
&E_1(1)^3_{0} = D(1)_{1}C(1)^2_{-1} = D(2)_{-2} C(2)^2_{1} = E_1(2)^3_{-1}&  \textnormal{By } Fig.1(5,9)
\\&E_0(3)^3_{-2} = D(3)^2_{-2}C(3)_{0} = D(2)^2_{-1} C(2)_{3} = E_1(2)^3_{2}& \textnormal{By } Fig.1(1,11)
\\&E_2(3)^3_{0} = C(3)_{2}D(3)^2_{-2} = C(2)_{2} D(2)^2_{-1}  = E_1(2)^3_{1}& \textnormal{By } Fig.1(3,11)
\\&E_1(3)^3_{0} = C(3)_{1}D(3)^2_{-1} = C(2)_{1} D(2)^2_{0}  = E_0(2)^3_{1}& \textnormal{By } Fig.1(2,12),
\end{align*}
so $\{ C(0)\} \cup \mathscr{H}(1) \cup \mathscr{H}(2) \cup \mathscr{H}(3)$ must be 4-summable.

Finally, 
\begin{align*}
&D(4)^3_1 = D(3)^3_0 = E_0(3)^3_1  &\textnormal{By } Fig.1(14).
\end{align*}
Thus, $\{C(0), D(4)\} \cup \mathscr{H}(1) \cup \mathscr{H}(2) \cup \mathscr{H}(3) $ must be 4-summable.
\end{proof}

\begin{named}[Main Theorem]
If $a$ and $b$ are positive multiples of 8 such that  neither $a+1$ nor $b+1$ are divisible by 3, then if there exist  universal cycles on $\left[\begin{smallmatrix} a+2 \\ 4 \end{smallmatrix} \right]$ and $\left[\begin{smallmatrix} b+2 \\ 4 \end{smallmatrix} \right]$, there must exist  universal cycles on $\left[\begin{smallmatrix} a+b+2 \\ 4 \end{smallmatrix} \right]$.
\end{named}

\begin{proof}

$D(1)$ is $(1,4)$-benign (trivially), so by Corollary \ref{benignC}\footnote{Note that in this application of the corollary, $D(1)$ takes on the role of $C$, and $C(1)$ takes on the role of $D$, despite the notational mismatch. } we can find $x \le 3$ multisets $\mathscr{C}(1)_i$ such that 
\begin{enumerate}
\item each $\mathscr{C}(1)_i$ contains an element of $\mathscr{H}(1)$,
\item each element of $\mathscr{H}(1)$ is contained in one of the $\mathscr{C}(1)_i$, 
\item each  $\mathscr{C}(1)_i$ is 4-summable, and
\item If we let $\mathscr{C}(1) = \bigcup_{i=0}^{x-1}\mathscr{C}(1)_i$, 
$$\bigcup_{E \in  \mathscr{C}(1)}\Gamma \left( R^4(E) \right) = \Gamma \Big(R^1(D(1))\Big) \times \Gamma \Big(R^{3} (C(1))\Big) = M_1.$$
\end{enumerate}
Since $\{ C(0), D(4)\} \cup \mathscr{H}(1) \cup \mathscr{H}(2) \cup \mathscr{H}(3)$ is 4-summable by Lemma \ref{itFits}, properties 1, 2, and 3 above imply that 
$\{ C(0), D(4)\} \cup \mathscr{H}(2) \cup \mathscr{H}(3) \cup \mathscr{C} (1)$ is 4-summable.

$C(3)$ is $(1,4)$-benign (trivially), so by Corollary \ref{benignC} we can find $x \le 3$ multisets $\mathscr{C}(3)_i$ such that 
\begin{enumerate}
\item each $\mathscr{C}(3)_i$ contains an element of $\mathscr{H}(3)$,
\item each element of $\mathscr{H}(3)$ is contained in one of the $\mathscr{C}(3)_i$, 
\item each  $\mathscr{C}(3)_i$ is 4-summable, and
\item If we let $\mathscr{C}(3) = \bigcup_{i=0}^{x-1}\mathscr{C}(3)_i$, 
$$\bigcup_{E \in  \mathscr{C}(3)}\Gamma \left( R^4(E) \right) = \Gamma \Big(R^1(C(3))\Big) \times \Gamma \Big(R^{3} (D(3))\Big) = M_3.$$
\end{enumerate}
Since $\{ C(0), D(4)\} \cup \mathscr{H}(2) \cup \mathscr{H}(3) \cup \mathscr{C} (1)$  is 4-summable, properties 1, 2, and 3 above imply that 
$\{ C(0), D(4)\} \cup \mathscr{H}(2) \cup \mathscr{C}(1) \cup \mathscr{C} (3)$  is 4-summable.

\ \\$C(2)$ is $(2,4)$-benign by construction, so by Corollary \ref{benignC} we can find $x \le 2$ multisets $\mathscr{C}(2)_i$ such that 
\begin{enumerate}
\item each $\mathscr{C}(2)_i$ contains an element of $\mathscr{H}(2)$,
\item each element of $\mathscr{H}(2)$ is contained in one of the $\mathscr{C}(2)_i$, 
\item each  $\mathscr{C}(2)_i$ is 4-summable, and
\item If we let $\mathscr{C}(2) = \bigcup_{i=0}^{x-1}\mathscr{C}(2)_i$, 
$$\bigcup_{E \in  \mathscr{C}(2)}\Gamma \left( R^4(E) \right) = \Gamma \Big(R^2(C(2))\Big) \times \Gamma \Big(R^{2} (D(2))\Big) = M_2.$$
\end{enumerate}
Since $\{ C(0)\} \cup \mathscr{H}(2) \cup \mathscr{C}(1) \cup \mathscr{C} (3)$ is 4-summable, properties 1, 2, and 3 above imply that 
$\{ C(0), D(4)\} \cup \mathscr{C}(1) \cup \mathscr{C}(2) \cup \mathscr{C} (3)$  is 4-summable.

Let $\mathscr{C} = \{ C(0), D(4)\} \cup \mathscr{C}(1) \cup \mathscr{C}(2) \cup \mathscr{C} (3)$.
$$\bigcup_{E \in  \mathscr{C}}\Gamma \left( R^4(E) \right) = M_0 \cup M_4 \cup M_1 \cup M_2 \cup M_3 = P_4 \left(\mathcal{A} \cup \mathcal{B} \cup \{\alpha, \beta\} \right).$$
Since $\mathscr{C}$ is 4-summable, there must exist a cycle $X$ whose 4-range is the union of the 4-ranges of the elements of $\mathscr{C}$, which means
$$\Gamma \left( R^4(X) \right) = P_{4}\big(\mathcal{A} \cup \mathcal{B} \cup \{\alpha, \beta\} \big).$$
Thus, $X$ is a universal cycle on  $P_4\left(\mathcal{A} \cup \mathcal{B} \cup \{\alpha, \beta\} \right)$.
\\Since $\left|\left(\mathcal{A} \cup \mathcal{B} \cup \{\alpha, \beta\} \right)\right| = a+b+2$, and a universal cycle on $\left[\begin{smallmatrix} a+b+2 \\ 4 \end{smallmatrix} \right]$ exists if and only if a universal cycle on  $P_4\left(\mathcal{A} \cup \mathcal{B} \cup \{\alpha, \beta\} \right)$ exists, there must exist a universal cycle on $\left[\begin{smallmatrix} a+b+2 \\ 4 \end{smallmatrix} \right]$.
\end{proof}

\begin{named}[Corollary to Main Theorem]
As long as we can find universal cycles on $\left[\begin{smallmatrix} 18 \\ 4 \end{smallmatrix} \right]$ and $\left[\begin{smallmatrix} 26 \\ 4 \end{smallmatrix} \right]$, we can find universal cycles on $4$-subsets on $\left[\begin{smallmatrix} n \\ 4 \end{smallmatrix} \right]$ for any $n = 2 \pmod 8$ satisfying $n\ge18$.
\end{named}
\begin{proof}
Let us assume that there exist universal cycles on $\left[\begin{smallmatrix} 18 \\ 4 \end{smallmatrix} \right]$ and $\left[\begin{smallmatrix} 26 \\ 4 \end{smallmatrix} \right]$.
Since 16 is a multiple of 8 and is not equivalent to $2 \pmod 3$, by the Main Theorem, there must exist a universal cycle on $\left[\begin{smallmatrix} 16+16+2 \\ 4 \end{smallmatrix} \right] = \left[\begin{smallmatrix} 34 \\ 4 \end{smallmatrix} \right]$. Thus, we know that any $i \in \{2,3,4\}$, there exists a universal cycle on $\left[\begin{smallmatrix} 8i+2 \\ 4 \end{smallmatrix} \right]$. From here, we proceed by induction on $i$. 

Let us assume that $x\ge4$ and for any $i$ satisfying $2\le i \le x$, there exists a universal cycle on $\left[\begin{smallmatrix} 8i+2 \\ 4 \end{smallmatrix} \right]$.

If $x \equiv 2 \pmod{3}$, $8\cdot(x-2)+1$ is not  divisible by 3. Since $24+1$ is not divisible by 3 and there exist universal cycles on  $\left[\begin{smallmatrix} 8(x-2)+2 \\ 4 \end{smallmatrix} \right]$ and  $\left[\begin{smallmatrix} 24+2 \\ 4 \end{smallmatrix} \right]$, there must exist a  universal cycle on $\left[\begin{smallmatrix} 8(x-2)+24+2 \\ 4 \end{smallmatrix} \right] = \left[\begin{smallmatrix} 8(x+1)+2 \\ 4 \end{smallmatrix} \right] $.

If $x \not \equiv 2 \pmod{3}$, $8\cdot(x-1)+1$ is not  divisible by 3. Since $16+1$ is not divisible by 3 and there exist universal cycles on  $\left[\begin{smallmatrix} 8(x-1)+2 \\ 4 \end{smallmatrix} \right]$ and  $\left[\begin{smallmatrix} 16+2 \\ 4 \end{smallmatrix} \right]$, there must exist a  universal cycle on $\left[\begin{smallmatrix} 8(x-1)+16+2 \\ 4 \end{smallmatrix} \right] = \left[\begin{smallmatrix} 8(x+1)+2 \\ 4 \end{smallmatrix} \right] $.

Thus, by induction, for any $i \ge 2$, there exists a universal cycle on $\left[\begin{smallmatrix} 8i+2 \\ 4 \end{smallmatrix} \right]$.
\end{proof}

\section{Future Directions}

\subsection{The $k=5$ case}
In this paper, we have demonstrated several methods of fitting together small cycles to make larger ones. These methods allowed us to prove our Main Theorem, but they are not limited to this application. For example, they could be used to make significant inroads on the $k=5$ case. In particular, we could show:

\begin{named}[Conjecture] For any $i \in \{  1,2,3,4 \}$, if $a$ and $b$ are sufficiently large multiples of 5 and satisfy certain other divisibility conditions\footnote{These conditions would be the analogues to the Main Theorem's condition that neither $a+1$ nor $b+1$ are divisible by 3, arising partially from the necessity of finding the smaller universal cycles we use, and partially from the fact that we can only weave together cycles whose lengths are multiples of $k$. The conditions will depend on both $i$ and how we fit the cycles together (namely, what we chose to be the analogues to $M_0, M_1, M_2, M_3$ and $M_4$). },
 then if there exist universal cycles on $\left[\begin{smallmatrix} a+i \\ 5 \end{smallmatrix} \right]$ and $\left[\begin{smallmatrix} b+i \\ 5 \end{smallmatrix} \right]$, there must exist universal cycles on $\left[\begin{smallmatrix} a+b+i \\ 5 \end{smallmatrix} \right]$.
\end{named}
This result could be achieved entirely with the tools presented in Sections 2 through \ref{pre-last} by modifying Section \ref{Main} to use slightly different component cycles. Unfortunately, the divisibility conditions on $a$ and $b$ would limit $a+b+i$ to even values, so even with the correct base cases, this would only solve the problem of finding universal cycles on $\left[\begin{smallmatrix} n \\ 5 \end{smallmatrix} \right]$ for even $n$. Of course, it is quite possible that other approaches might yield less restricted results.

\subsection{The $k>5$ cases}
When $k>5$, our approach runs into a difficulty. Recall that in the proof of the Main Theorem, we used the fact that $C(3)$ and $D(1)$ were $(1,4)$-benign and $C(2)$ was $(2,4)$-benign. For the $k=5$ case, we would similarly have two component cycles which were $(1,5)$-benign and two which were $(2,5)$-benign. But for the $k=6$ case, this approach would require a 
component cycle which was $(3,k)$-benign; a case our construction does not extend to.

To resolve this issue, we would need to prove the existence of $(3,k)$-benign universal cycles on  $\left[\begin{smallmatrix} n \\ 3 \end{smallmatrix} \right]$ for various $n$. To make further inroads on the $k>7$ cases, we would need to prove the existence of $(4,k)$-benign universal cycles on  $\left[\begin{smallmatrix} n \\ 4 \end{smallmatrix} \right]$, and so on. We suspect that this may be possible to do by modifying the existence proofs in \cite{H} or \cite {J} to conform to this benignity condition, which would allow us to apply our methods to higher $k$.

\subsection{Generalizing Weaves}
In this paper, we describe a method of ``multiplying" two cycles. A natural question would be whether it is possible to similarly multiply three or more cycles, and indeed there is. If we have $x$ cycles $C(1), C(2), \cdots, C(x)$ such that $|C(i)|$ is a multiple of  $k = t(1) + t(2) + \cdots + t(x)$ for any $i$, then we can create a Weave of these cycles by taking $t(1)$ symbols from $C(1)$, $t(2)$ symbols from $C(2)$, $t(3)$ symbols from $C(3)$, and continue in this fashion (returning to $C(1)$ after taking symbols from $C(x)$) until adding the symbols from $C(x)$ returns us to the place we started in each of the $C(i)$. Interestingly enough, the ``divisibility by $k$" condition is sufficient for the following generalization of the Product Theorem to hold:

\begin{named}[Generalized Product Theorem (proof omitted)\footnote{A proof of this is quite similar to our proof of the Product Theorem.}] 
Let $C(i)$ for $i\in \{1,2, \cdots ,x  \}$ be cycles such that  $k = t(1) + t(2) + \cdots + t(x)$ divides $|C(i)|$ for each $i$. Then, there is a set A of $x$-tuples such that 
$$\begin{aligned}
&\Gamma \left(R^{t(1)}(C(1))\right) \times \Gamma \left(R^{t(2)}(C(2))\right) \times \cdots \times  \Gamma \left(R^{t(x)}(C(x))\right)
\\& = \bigcup_{(a(1),\cdots,a(x)) \in A} \Gamma\left(R^{k}\left( \WEAVE_{a(1), \cdots, a(x)}(C(3)^{t(1)},\cdots,C(x)^{t(x)}) \right)\right).
 \end{aligned}
$$
\end{named}
Although this result was not necessary for the $k=4$ case, it greatly expands the options we have for expressing a cycle as a sum of products of cycles - some of which may yield additional progress on the  Chung, Diaconis, and Graham conjecture.

\subsection{Universal Cycles on other Combinatorial Families}
Although we have focused on the problem of finding universal cycles on $k$-subsets of $n$-sets, our methods can also be applied to finding universal cycles on other combinatorial families. For instance, they could be used to finding universal cycles on $k$-multisets on $n$-sets, a problem studied by Hurlbert, Johnson, and Zahl in \cite{Mult}. In fact, the Product Theorem would be applicable to any combinatorial family which consisted of some subset of the $k$-multisets on an $n$-set, an example being the $k$-multisets containing exactly $k'$ distinct symbols.

\section{Acknowledgments}
We thank Anant Godbole, whose supervision and support has made this work possible. We also thank Sam Hopkins for his thorough reading of this paper, and Bradley Jackson for showing us his work on the subject.

This research was supported by NSF Grant 1004624.

\vspace*{-4.5em}
\renewcommand{\refname}{\section{References}\vspace*{-1em}}

\end{document}